\documentclass[12pt,final]{amsart}

\usepackage{amssymb, amsmath, amsthm}
\usepackage[english]{babel}
\usepackage{mathrsfs} 
\usepackage{cancel}

\usepackage[text={5.5in,8in},centering]{geometry}

\usepackage[dvipsnames]{xcolor}
\usepackage{etoolbox}
\patchcmd{\section}{\normalfont}{\normalfont\color{darkblue}}{}{}
\patchcmd{\subsection}{\normalfont}{\normalfont\color{darkblue}}{}{}
\patchcmd{\subsubsection}{\normalfont}{\normalfont\color{darkblue}}{}{}

\usepackage{bm}
\usepackage{subfig,tikz}
\usepackage{wrapfig}
\usepackage{color}
\usepackage{xspace}
\usetikzlibrary{shapes.geometric}

\definecolor{darkblue}{rgb}{0.0, 0.0, 0.55}

\usepackage{graphicx}
\usepackage[colorlinks=true, citecolor=green, linkcolor=red, urlcolor=black]{hyperref}
 \usepackage{xcolor}
\usepackage{hyperref}

\newtheorem{lemma}{Lemma}
\newtheorem{theorem}{Theorem}
\newtheorem{corollary}{Corollary}

\theoremstyle{remark}
 
\newtheorem{remark}{Remark} 
\theoremstyle{theorem}

\def\R{\mathbb R}

\def\N{\mathbb N}

\def\p{\partial}

\DeclareMathOperator{\supp}{supp}
\DeclareMathOperator{\dist}{dist}

\newcommand{\norm}[1]{\left\|#1 \right\|}

\renewcommand{\tilde}{\widetilde}
\renewcommand{\leq}{\leqslant}
\renewcommand{\geq}{\geqslant}
\title[recovering time-dependent coefficients with one measurement]{Global Recovery of a time-dependent coefficient for the Wave equation from a single measurement}
\author[A. Feizmohammadi]{Ali Feizmohammadi}
\address{Department of Mathematics, University College London, 
Gower Street, London UK, WC1E 6BT.}
\email{a.feizmohammadi@ucl.ac.uk}
\author[Y. Kian]{Yavar Kian}
\address{Department of Mathematics, Aix Marseille Universit\'{e}, Universit\'{e} de Toulon, CNRS, CPT, Marseille, France}
\email{yavar.kian@univ-amu.fr}
\begin{document}

\begin{abstract}
We consider the formally determined inverse problem of recovering an unknown time-dependent potential function from the knowledge of the restriction of the solution of the wave equation to a small subset, subject to a single external source. We show that one can determine the potential function, up to the natural obstruction for the problem, by using a single source placed in the exterior of the spacetime domain and subsequently measuring the solution in a small neighborhood outside of the spacetime domain. The approach is based on considering a dense collection of light rays and constructing a source function that combines a countable collection of sources that each generates a wave packet near a light ray in the collection. We show that measuring the solution corresponding to that single source simultaneously determines the light ray transform along all the light rays in the collection. The result then follows from injectivity of the light ray transform. Our proof also provides a reconstruction algorithm.
\end{abstract}
\maketitle
{
  \hypersetup{linkcolor=black}
  \tableofcontents
}
%\tableofcontents
%---------------------------------------------------------------
%% DOCUMENT %%%%%%%%%

\section{Introduction and outline of the method}
\label{intro}
Let $T>0$ and $\Omega\subset\subset\tilde{\Omega}$ be domains in $\R^n$ with smooth boundaries. We assume $n\geq 2$. Given any 
$$ f\in L^2(\R^{1+n})\quad \text{with} \quad \supp f \subset (0,T)\times (\tilde{\Omega}\setminus\overline{\Omega}),$$
we consider the wave equation
 \begin{equation}\label{pf}
\begin{aligned}
\begin{cases}
(\Box+V(t,x))u=f(t,x),
& (t,x) \in (0,T)\times \R^n,
\\
u(0,x)=\p_tu(0,x)=0,\,
& x \in \R^n
\end{cases}
    \end{aligned}
\end{equation}
where $\Box=\partial_t^2-\Delta_x$ is the wave operator and $V\in L^\infty((0,T)\times \Omega)$ is an a priori unknown function. This problem admits a unique solution $u$ in the energy space
\begin{equation}
\label{sobolev_space}
\mathcal C^1([0,T];L^2(\R^n)) \cap \mathcal C([0,T];H^1(\R^n)).
\end{equation}
Moreover, $u(t,\cdot)$ is compactly supported for each $t \in [0,T]$ and the following bounds hold:
\begin{equation}
\label{energy}
\norm{u}_{\mathcal C^1(0,T;L^2(\R^n))} + \|u\|_{\mathcal C(0,T;H^1(\R^n))}\leq C \|f\|_{L^2((0,T)\times (\tilde{\Omega}\setminus\overline{\Omega}))},
\end{equation}
where $C$ is a positive constant depending on the geometry and $\|V\|_{L^2((0,T)\times \Omega)}$.

In the present paper we consider the following natural inverse problem; Does there exist a universal source function $f \in L^2((0,T)\times (\tilde{\Omega}\setminus\overline{\Omega}))$, only depending on $T$, $\Omega$ and $\tilde \Omega$, such that the knowledge of $u$ restricted to $(0,T)\times \mathcal O$, with $\mathcal O\subset \tilde{\Omega}\setminus\overline{\Omega}$ an open subset, determines uniquely the unknown potential $V$?

\subsection{Main results}
There is a natural obstruction to uniqueness for the potential $V$. Namely, due to finite speed of propagation for the wave equation, the knowledge of $u_{|(0,T)\times(\tilde{\Omega}\setminus\overline{\Omega})}$ contains no information about the potential on the set
$$ \{ (t,x) \in (0,T)\times \Omega\,|\, 0<t<\dist{(x,\p \Omega)}\quad \text{or}\quad T-\dist{(x,\p\Omega)}<t<T\}.$$
We refer the reader to \cite[Section 1.1]{Kian1} for more details. Thus, the optimal domain for recovering the potential function will be the complement of this set that is given by
$$\mathcal D=  \{(t,x) \in (0,T)\times \Omega\,|\, \dist{(x,\p \Omega)}<t<T-\dist{(x,\p\Omega)}\}.$$

This paper is concerned with the resolution of the question posed above in the optimal set $\mathcal D$. We remark that the main complexity of this single source inverse problem stems from the fact that it is a formally determined inverse problem. Heuristically, given any fixed source function $f$, the solution $u$ to \eqref{pf} and the unknown potential $V$ are both functions of $1+n$ variables. Our main result can be stated as follows.

\bigskip
\begin{theorem}
\label{t0}
Let $\Omega \subset\subset \tilde\Omega$ be domains in $\R^n$ with smooth strictly convex boundaries and let $T>$Diam$(\Omega)$. Then there exists a function $f \in L^2(\R^{1+n})$, with $\supp f \subset (0,T)\times (\tilde{\Omega}\setminus\overline{\Omega})$, such that given any 
\begin{equation}\label{V_space} V_j \in \mathcal C^4([0,T]\times\R^n)\cap \mathcal C([0,T];\mathcal C^4_0(\Omega)),\quad j=1,2,\end{equation} 
the following injectivity result holds,
\begin{equation}\label{t0a}u_1=u_2,\quad \text{on $(0,T)\times(\tilde{\Omega}\setminus\overline{\Omega})$}\quad\Longrightarrow \quad V_1=V_2 \quad \text{on $\mathcal D$}.\end{equation}
Here, $u_j$, $j=1,2$, is the unique solution to the wave equation \eqref{pf} in energy space \eqref{sobolev_space} subject to $V=V_j$ and source term $f$. 
\end{theorem}
Note that the result of Theorem \ref{t0} is stated with a single measurement on a neighborhood of the lateral boundary $(0,T)\times\p \Omega$ of the solution of \eqref{pf} subjected to our universal source $f$. As a direct consequence of Theorem \ref{t0}, we can show that, when $T$ is large enough, it is possible to recover uniquely the coefficient $V$ on some subset of $\mathcal D$ from a single measurement on a neighborhood of the lateral boundary $(0,T)\times\gamma$ with $\gamma$ an arbitrary open subset of $\p \Omega$. This result can be stated as follows.
\begin{corollary}\label{c1} Let the condition of Theorem \ref{t0} be fulfilled, fix  $f \in L^2(\R^{1+n})$ the universal source introduced in Theorem \ref{t0} and assume that $\tilde{\Omega}\setminus \overline{\Omega}$ is connected. Consider $\mathcal O$ an arbitrary open subset of $\tilde\Omega\setminus\Omega$, $T_1>$Diam$(\Omega)$ and $\mathcal D_{T_1}$ a subset of $(0,T)\times\Omega$ given by
$$\mathcal D_{T_1}=  \{(t,x) \in (0,T_1)\times \Omega\,|\, \dist{(x,\p \Omega)}<t<T_1-\dist{(x,\p\Omega)}\}.$$
Assume that the following condition is fulfilled
\begin{equation}\label{c1a} T>T_1+\sup_{x\in\overline{\tilde{\Omega}}\setminus \Omega}\textrm{\emph{dist}}(x,\mathcal O),\end{equation}
where \emph{dist} denotes   the distance function on $\tilde{\Omega}\setminus \overline{\Omega}$. Then, for any $V_j$ in the Sobolev space \eqref{V_space}, $j=1,2$, and for $u_j$ solving \eqref{pf} with $V=V_j$ and source term $f$, there holds, 
\begin{equation}\label{c1a}u_1=u_2,\quad \text{on $(0,T)\times\mathcal O$}\quad\Longrightarrow \quad V_1=V_2 \quad \text{on $\mathcal D_{T_1}$}.\end{equation}
\end{corollary}
                                                                
\subsection{Previous literature}

The recovery of coefficients appearing in hyperbolic equations from boundary measurements, or the so-called Dirichlet-to-Neumann map, is an inverse problem with a rich recent literature. It physically arises in the study of recovery of information about signal propagation, such as determining the evolving density of an in-homogeneous medium or determining the wave speed of sound propagating in different layers of earth. It is also related to the challenging inverse problem of determining non-linear terms in hyperbolic equations (see e.g. \cite{Kian3}). These non-linear questions are motivated in part by the study of vibrating systems or the detection of perturbations arising in electronics, such as the telegraph equation or the study of semi-conductors (see for instance \cite{CH}). 

Broadly speaking, the literature of inverse problems for hyperbolic equations can be divided into two categories, namely that of recovering time-independent or time-dependent coefficients, and the majority of the literature in both cases uses infinite measurements. Here, by infinite measurements we mean that an infinite number of sources $f$ in \eqref{pf} are required to deduce uniqueness of the coefficient $V$.

We begin with reviewing the literature of uniqueness results with infinite measurements. In the time-independent category, the first class of uniqueness results were obtained in the works \cite{Bel, RS1,I1}. We mention also the subsequent works \cite{AS,BCY,Ki,SU1,SU2} that also provide stability estimates from full or partial knowledge of the hyperbolic Dirichlet-to-Neumann map. In particular, the approach of \cite{Bel} is based on the discovery of the powerful boundary control method. At its core, this method is based on combining controllability theory and unique continuation for the wave equation together with boundary integral identities. This approach even extends to the recovery of a Riemannian manifold, up to isometry, from boundary measurements for the wave equation with a variable coefficient principal part. We refer to \cite{BelK,KKL} for applications of the boundary control method to the  recovery of a Riemannian manifold. This method also allows unique recovery of coefficients in the case where the sources and the receivers are located on disjoint sets, see for example \cite{KKLO,LO}. 

In the case of time-dependent coefficients, the boundary control method is less successful, even when the principal part of the wave equation has constant coefficients, as in \eqref{pf} for example. Indeed, the method relies on the unique continuation result of Tataru \cite{Ta}, that fails to hold in general, unless the time-dependence of all the coefficients is real-analytic (see the general  counter examples of \cite{AB}). In the case that the coefficients depend analytically on the time variable we mention the works \cite{E1,E2,E3} where the author extended the boundary control method to these class of coefficients. 

For more general time-dependent coefficients, the approach of \cite{RS1,St} based on the construction of geometric optics solutions, has been successful in deriving uniqueness and stability results. These results and many of the subsequent works are based on the principle of propagation of singularities for the wave equation and extend to the case of variable coefficient wave equations, where the problem of recovering coefficients reduces to injectivity of certain geometrical data on Lorentzian manifolds, see for example \cite{FIKO,FIO,FK,KO,LOSG}. We remark that all of these works require strong geometrical assumptions and that in general recovering time-dependent coefficients for variable coefficient wave equations remains a daunting prospect.

All of the aforementioned results are stated with infinitely many measurements (or sources). As discussed above and specifically in the case of the wave equation with constant coefficient principal part as in \eqref{pf}, the recovery of coefficients has been well-understood both in the time-dependent or time-independent categories. 

The story is vastly different when one considers a finite number of measurements, where there seems to be no result for recovering a time-dependent coefficient. In the time-independent category however, by applying the Bukhgeim-Klibanov approach  of \cite{BK} that is based on Carleman estimates, some authors have considered the recovery of time-independent coefficients from a single measurement, see for example \cite{Klibanov0}. Since then this approach has been improved to include stability results by several authors. We refer the reader to the works of \cite{Bell,BY,IY,Ya,Ste2} for further results in this direction. 

The Bukhgeim-Klibanov approach is based on linearizing the inverse problem and reformulating the problem into that of recovering a source term. In light of this, all the results obtained by this approach require a non-vanishing initial condition for the solution $u$. The presence of this non-vanishing initial condition corresponds to some a priori information on the inaccessible part (the part $x\in\Omega$) that makes these results more difficult to apply in reality. 

As an alternative to the Bukhgeim-Klibanov approach, we also mention the works  \cite{AS,CY,HLYZ,KLLY} where the authors considered an approach based on the construction of suitable input for proving recovery of time-independent coefficients appearing in diffusion equations. Note that the approach of \cite{AS,CY,KLLY} is based on the analyticity in time of the solution which does not hold for hyperbolic equations. 

Within the time-independent category, a few authors have also considered approaches based on a single measurement of the solution to the wave equation subjected to a point source, represented by a Dirac delta distribution, on the boundary or inside the domain. In contrast to the natural energy space \eqref{sobolev_space} for \eqref{pf} that we consider in this paper, these works are based on extending the solution space to \eqref{pf} in a distributional sense to allow very singular sources. In this setting one of the first results that we can mention is the one of \cite{SS} where partial information about the coefficient of a hyperbolic equation can be recovered from a single measurement associated with a boundary point source. In \cite{RY} the authors proved that under an additional smallness assumption on the unknown coefficient, it is possible to stably recover it from a single boundary measurement of the solution subjected to an internal point source. In the same spirit, the works of \cite{Ra1,Ra2,RSacks} were devoted to the unique recovery of a special class of zeroth order time-independent coefficients. 

Finally, we mention the work of \cite{HLO} where the recovery of a time-independent Riemannian metric is considered from a single measurement. There, the single measurement corresponds to a source term that is the sum of a countable number of Dirac delta distributions in time and space.

\subsection{A comparison with the previous literature}

Let us now discuss the novelties of our main result. Firstly, and to the best of our knowledge, Theorem \ref{t0} corresponds to the first result for unique recovery of a general time-dependent coefficient from a single measurement subject to the wave equation, or any other evolution PDE. 

In fact even within the class of time-independent coefficients, Theorem \ref{t0} appears to be the first single measurement uniqueness result that does not require initial time excitation of solutions, and that also provides a source function that is compatible with the natural energy class \eqref{sobolev_space} for solutions of the wave equation with vanishing initial conditions. In view of these features, even for time-independent coefficients, the statement of our uniqueness result can make it more suitable for applications.

As a second novelty, we mention that Theorem \ref{t0} proves uniqueness of time-dependent zeroth order coefficients in the optimal region $\mathcal D$. Even in case of infinite measurements, most of the uniqueness results for time-dependent coefficients either require information at $t=0,T$, see for example \cite{Kian,Kian1,Kian2,KO}, or require the knowledge of the coefficient outside of $\mathcal D$, \cite{Ben,FIKO,FK,RR}. In the latter group, uniqueness results are usually provided on a sub-optimal region that is approximately equal to
\begin{equation}\label{MU}\left\{(t,x)\in\left(\frac{d}{2},T-\frac{d}{2}\right)\times\Omega:\ \textrm{dist}(x,\p\Omega)<\min\left(t-\frac{d}{2},T-t-\frac{d}{2}\right)\right\},\end{equation}
with $d=\textrm{Diam}(\Omega)$ and for $T>d$.

Finally, we also mention that Corollary \ref{c1} provides a partial data version of our main result, where the measurements associated to the single source are restricted to a neighborhood of an arbitrary portion of the boundary $\partial\Omega$, provided that the time $T$ is large enough. In particular, for coefficients that are real analytic with respect to the time variable, the result of Corollary \ref{c1} corresponds to the full recovery of the coefficient on the full spacetime domain $(0,T)\times \Omega$. Thus, Corollary \ref{c1} can also be viewed as a single boundary measurement formulation, in terms of localization of the measurement, of the work of \cite{E1,E2} that is devoted to the recovery of time analytic coefficients from infinitely many measurement on an arbitrary portion of the boundary.

Our proof explicitly constructs the universal source function $f$ and also provides an algorithm for reconstructing $V$. We remark that the domain $\tilde \Omega$ in the statement of Theorem~\ref{t0} could be as small as one wishes, or in other words the source $f$ can be supported in a very small neighborhood of $(0,T)\times\p \Omega$. We believe that the approach here could be pushed in principle to allow unique recovery of time-dependent coefficient by a single {\em boundary} measurement as well, instead of measurements that are associated to a source located near the boundary. We leave this, as well as the extension of our result to the setting of Lorentzian manifolds, as directions for future research. 

\subsection{Outline of the paper}

Let us briefly sketch the methodology employed in proving Theorem~\ref{t0}. We recall first that the term {\em light ray} refers to a curve in spacetime that is a geodesic with respect to the Minkowski metric and whose tangent vector at each point along the curve is null. We will start with a countable collection of light rays that densely pack the domain $\mathcal D$. Precisely, this means that given any small positive $\epsilon$ and any light ray $\gamma$ in $\mathcal D$, there will be a light ray in the collection that stays within a distance $\epsilon$ of $\gamma$. Next, we will consider a universal source function that is constructed based on combining infinitely many source functions that each generates a geometric optic solution to \eqref{pf} concentrating along a light ray in the collection. We show that the solution to \eqref{pf} corresponding to this universal source determines the integrals of the unknown function $V$ along all the light rays in the collection (see Theorem~\ref{t1}). The main theorem then follows by using the density of the rays in the collection and injectivity of the light ray transform, see for example \cite{Be,Ste}.

The paper is organized as follows. In Section~\ref{prelim_sec}, we begin with introducing a few notations used in the paper and then define the admissible collection of light rays that tightly pack the spacetime domain. Section~\ref{geo_optics_sec} is concerned with a review of the classical geometric optics solutions to \eqref{pf} also known as wave packets. The construction of wave packets in this paper is modified to allow thinner supports for these solutions as the frequency increases. Next, we show that it is possible to construct explicit sources that are supported in the set $(0,T)\times (\tilde{\Omega}\setminus \overline{\Omega})$ and such that the solution to \eqref{pf} subject to these source functions generate the desired wave packets. In section~\ref{universal_source_sec} we construct the universal source function $f$ that combines the geometric optic solutions via a double infinite summation corresponding to the set of light rays and the set of frequencies of the geometric optic solutions associated to each light ray. Section~\ref{asymptotic_analysis_sec} is concerned with the proof of Theorem~\ref{t1}, showing that the knowledge of $u_{|(0,T)\times (\tilde{\Omega}\setminus \overline{\Omega})}$, with $u$ solving \eqref{pf}, uniquely determines the integrals of $V$ along all the light rays in the collection. The proof of the main theorem follows immediately from combining Theorem~\ref{t1} and injectivity of the light ray transform. This is sketched in Section~\ref{proof_section}, where we also prove Corollary \ref{c1}.

\section{Preliminaries}
\label{prelim_sec}

\subsection{Notation}
Let us introduce a few notations that will be used in the paper. As already discussed, we use $$(t,x)=(t,x^1,x^2,\ldots,x^n)$$ for the spacetime coordinate system with $t \in \R$ and $x \in \R^n$. Given two vectors $v,w \in \R^{n}$, their inner product and norm is defined respectively by the expressions
$$ v\cdot w = \sum_{j=1}^n v_j\,w_j\quad\text{and}\quad |v|=\sqrt{v\cdot v}.$$
Throughout the paper we use the notation $\chi$ to stand for a smooth non-negative cutoff function satisfying 
\begin{equation}
\label{cutoff}
\|\chi\|_{L^2(\R)}=1\quad\text{and}\quad
\chi(t)= \begin{cases}
        1  & \text{if}\,\, |t|\leq\frac{1}{8\sqrt{n}},\\
             0  & \text{if}\,\,|t|\geq \frac{1}{4\sqrt{n}}.
     \end{cases}
\end{equation}
We denote also by $\mathbb N$ the set $\{1,2,\ldots\}$.
As already discussed in the introduction, the construction of the universal source function $f$ in this paper involves the summation of a countable number of smooth sources each of which generates a wave packet near a light ray. For this reason it is important to use a consistent notation for convergence of infinite series. Since we require $f \in L^2(\R^{1+n})$ with $\supp f \subset (0,T)\times (\tilde{\Omega}\setminus\overline{\Omega})$, we will be working with convergence of source terms in the $L^2((0,T)\times (\tilde{\Omega}\setminus\overline{\Omega}))$ topology. We formally write 
$$ f= \lim_{j\to \infty}f_j$$ 
to stand for convergence with respect to the $L^2((0,T)\times (\tilde{\Omega}\setminus\overline{\Omega}))$ topology of a sequence of functions $\{f_j\}_{j=1}^{\infty}\subset L^2((0,T)\times (\tilde{\Omega}\setminus\overline{\Omega}))$. For solutions to the wave equation \eqref{pf}, we will work with the natural Sobolev space \eqref{sobolev_space} and as such we formally write
$$ u =\lim_{j\to \infty}u_j$$
to stand for convergence with respect to the \eqref{sobolev_space} topology. We close this section by recording the following trivial lemma about convergence of solutions to the wave equation. We have included the proof for the sake of completeness.   
 \begin{lemma}
 \label{conv_lem}
 Let $\{f_j\}_{j=1}^{\infty} \subset L^2((0,T)\times (\tilde{\Omega}\setminus\overline{\Omega}))$ and assume that this sequence of sources converges to a source $f$ in this topology. Let $u_j$ denote the unique solution to \eqref{pf} with source $f_j$. Then, the sequence $\{u_j\}_{j=1}^{\infty}$ converges to a function $u$ with respect to the \eqref{sobolev_space} topology. Moreover, $u$ is the unique solution to \eqref{pf} subject to the source $f$.  
 \end{lemma}
 
 \begin{proof}
 Note that for each $j,k \in \N$, the function $u_j-u_k$ solves the wave equation with source function $f_j-f_k$. Therefore, the energy estimate \eqref{energy} applies to obtain
 $$\|u_j-u_k\|_{\mathcal C([0,T];H^1(\R^n))\cap \mathcal C^1([0,T];L^2(\R^n))} \leq C \|f_j-f_k\|_{L^2((0,T)\times (\tilde{\Omega}\setminus\overline{\Omega}))}.$$
 Therefore we deduce the the sequence $\{u_j\}_{j=1}^{\infty}$ is a Cauchy sequence with respect to the \eqref{sobolev_space} topology. We now define
 $$ u = \lim_{j\to \infty} u_j$$
 and proceed to prove that $u$ satisfies \eqref{pf} with source term $f$. The initial conditions are clearly satisfied. To prove $(\Box+V) u =f$, it suffices to show that
 $$ \int_{(0,T)\times\R^{n}}f\,v\,dx=\int_{(0,T)\times\R^{n}} a(u,v)\,dx\quad \quad \forall\, v \in \mathcal C^{\infty}_0([0,T]\times\R^{n})$$
 where $$a(u,v)=-\frac{\p u}{\p t}\,\frac{\p v}{\p t}+\sum_{k=1}^n\frac{\p u}{\p x^k}\,\frac{\p v}{\p x^k }+V\,u\,v.$$
 We assume without loss of generality that $\|v\|_{H^1((0,T)\times\R^{n})}=1$ and note that given any $j \in \N$:
 \begin{multline*}
\left|\int_{(0,T)\times\R^{n}}(f\,v-a(u,v))\,dx\right|= \left|\int_{(0,T)\times\R^{n}}((f-f_j)\,v+a(u_j-u,v))\,dx\right|\\
< C\,(\|f-f_j\|_{L^2((0,T)\times (\tilde{\Omega}\setminus\overline{\Omega}))} + \|u-u_j\|_{ \mathcal C(0,T;H^1(\R^n))}\\
+\|u-u_j\|_{\mathcal C^1(0,T;L^2(\R^n))}).
 \end{multline*}
 The proof is completed since $f_j\to f$ and $u_j\to u$ in their respective topologies.
 \end{proof}
 
\subsection{Constructing a countable dense set of light rays}
\label{ray_sec}
The aim of this section is to construct a countable family of light rays that tightly pack the set $(0,T)\times\Omega$ and also introduce some notation that will be used later in the paper. In what follows, a future pointing light ray is a curve $\gamma:\R\to\R^{1+n}$ given by the parametrization
$$ \gamma(s)= \gamma(0) + s\,(1,\xi)\quad s\in \R$$
for some unit vector $\xi \in \R^n$.
\\

Let $\mathcal T=\{t_j:\ j\in\mathbb N\}$ denote an ordering of the rational numbers in the interval $(0,T)$ and let $\mathcal P=\{p_j\in\partial\Omega: \ j\in\mathbb N\}$
denote a dense set of points on $\p\Omega$. We consider the countable set of all future pointing light rays $\gamma:\R \to \R^{1+n}$, parametrized as above, that satisfy the following three properties:
\begin{itemize}
\item[(i)]{The intersection of $\gamma$ with $(0,T)\times\Omega$ lies in the set $\mathcal D$.}
\item[(ii)]{The earliest intersection of $\gamma$ with $\R\times \p \Omega$ is the point $\gamma(0)\in \mathcal T\times \mathcal P$.}
\item[(iii)]{The projection of $\gamma$ onto the spatial coordinates is a straight line that contains two distinct points in $\mathcal P$.}
\end{itemize}
We consider an ordering of this countable set of light rays and denote it by $$\mathbb V=\{\gamma_j\}_{j=1}^{\infty}.$$
Given any $j \in \N$, we choose a point $$q_{j}=(s_j,x_j)=\gamma_j(\hat{s}_j) \in (0,T)\times (\tilde{\Omega}\setminus\overline{\Omega})$$ for some $\hat{s}_j<0$. Here, $s_j$ and $x_j$ are the time and spatial coordinates of $q_j$ respectively.  We also choose a constant $\delta_j>0$ such that   
$$B_{\delta_j}(q_{j})\subset (0,T)\times (\tilde{\Omega}\setminus\overline{\Omega})$$
and additionally that the intersection of $(0,T)\times \Omega$ with the tubular neighborhood of the ray $\gamma_j$ of radius $\delta_j$ lies in the set $\mathcal D$. Here $B_{\delta_j}(q_{j})$ denotes the ball of radius $\delta_j$ centered at the point $q_{j}$.  Since $\mathbb V$ is countable, we can always choose the sequence $\{\delta_j\}_{j=1}^{\infty}$ to be strictly decreasing, that is to say
$$ \delta_1>\delta_2>\delta_3>\ldots.$$
Next and for the purpose of later application, we define two smooth functions $\zeta_{j,{\pm}}:\R^{1+n}\to \R$ that satisfy
\[
\zeta_{j,-}(t)= \begin{cases}
        0  & \text{if}\,\, t\leq s_j-\frac{\delta_j}{4\sqrt{n}},\\
             1  & \text{if}\,\,t\geq s_j.
     \end{cases}
\]
and
\[
\zeta_{j,+}(t) = \begin{cases}
        0  & \text{if}\,\, t\geq s_j+\frac{\delta_j}{4\sqrt{n}},\\
				\ &\ \\
             1  & \text{if}\,\,t \leq s_j+\frac{\delta_j}{8\sqrt{n}}.
    \end{cases}
\]

Observe that since $\hat{s}_j<0$, it follows that $\zeta_{j,-}=1$ on the segment of the light ray $\gamma_{j}$ that lies inside the set $(0,T)\times \Omega$.

\section{Geometric optics}
\label{geo_optics_sec}

In this section, we fix $j\in \N$ and recall the geometric optics construction, with some modifications, for the wave equation 
$$ \Box u+V u=0$$
that gives solutions concentrating on the light ray $\gamma_{j} \in \mathbb V$. By definition of the set $\mathbb V$ there exists unique indices $k_j,l_j,m_j \in \N$ with $l_j\neq m_j$ such that the light ray $\gamma_j$ is given by the parametrization 
$$ \gamma_{j}(s)=(t_{k_j}+s,p_{l_j}+s\,\xi_j)\quad \text{for all} \quad s\in \R$$
where
$$ \xi_j = \frac{p_{m_j}-p_{l_j}}{|p_{m_j}-p_{l_j}|}.$$ 
Note that, by strict convexity of $\p \Omega$, the light ray $\gamma$ intersects the boundary $(0,T)\times \p\Omega$ precisely two times at the points $\gamma(0)$ and $\gamma(|p_{m_j}-p_{l_j}|)$. Moreover, by property (i) in the definition of $\mathbb V$, the light ray does not intersect the set $\{0,T\}\times\Omega$. \\

The geometric optics construction here is based on the ansatz
\begin{equation}
\label{ansatz}
\mathcal U_{j,\tau}(t,x)= e^{{\rm i}\tau (-t+\xi_j\cdot x)}v_{j,\tau}(x)=e^{{\rm i}\tau (-t+\xi_j\cdot x)}\left(\sum_{k=0}^2 \frac{v_{j,\tau}^{(k)}(x)}{\tau^k}\right)
\end{equation}
where $\tau>e$  is a parameter.
We write
\begin{equation}\label{conjugatego}
\begin{aligned}
&(\Box+V) (e^{{\rm i}\tau(-t+\xi_j\cdot x)}v_{j,\tau})= \\
&e^{{\rm i}\tau (-t+\xi_j\cdot x)} \left( -2{\rm i}\tau(\partial_tv_{j,\tau}+ \xi_j\cdot\nabla_x v_{j,\tau}) + (\Box+V) \,v_{j,\tau}\right).
\end{aligned}
\end{equation}
The amplitudes $v_{j,\tau}^{(0)}$, $v_{j,\tau}^{(1)}$ and $v_{j,\tau}^{(2)}$ are determined iteratively, based on the requirement that the expression \eqref{conjugatego} vanishes in powers of $\tau$ up to second order. In particular, this imposes the transport equation 
    \begin{align}\label{transp_a0}
\partial_t v_{j,\tau}^{(0)}+ \xi_j\cdot\nabla_xv_{j,\tau}^{(0)}=0
    \end{align}
on $v_{j,\tau}^{(0)}$. To solve this equation, we first choose the vectors $e_{j,1},\ldots,e_{j,n-1}\in\R^n$ such that 
$$\{\xi_j,e_{j,1},\ldots,e_{j,n-1}\}$$ form an orthonormal basis for $\R^n$. Next, we set
\begin{multline}\label{a_0}v_{j,\tau}^{(0)}(t,x)= \left(\frac{\log\tau}{\delta_j}\right)^{\frac{n}{2}}\chi[((\log\tau)\,\delta_j^{-1}(s_j-t+(x-x_j)\cdot\xi_j )]\times \\
\prod_{k=1}^{n-1}\chi[(\log\tau)\,\delta_j^{-1}(x-x_j)\cdot e_{j,k} ],
\end{multline}
where $\delta_j$, $s_j$, $x_j$ are as defined in Section~\ref{ray_sec} and the function $\chi$ is given by \eqref{cutoff}. Then (\ref{transp_a0}) holds and the amplitude $v_{i,\tau}^{(0)}(t,x)$ is supported in a tubular neighborhood of radius $\frac{\delta_j}{2\,\log\tau}$ around $\gamma_j$. We emphasize here that our construction of the leading amplitude is different from that of the classical geometric optic constructions, as the support of the geometric optic solution around $\gamma_j$ also depends on the frequency parameter $\tau$. Indeed, as $\tau$ grows, the support of the geometric optics also gets more localized around the light ray $\gamma_j$. This will be important in our analysis. Moving on, the subsequent terms $v^{(k)}_{j,\tau}$ with $k=1,2$ are constructed iteratively by solving the transport equations
\begin{equation}\label{transmin}
-2{\rm i} ( \partial_tv_{j,\tau}^{(k)}+\xi_j\cdot\nabla_xv_{j,\tau}^{(k)}) + (\Box+V) v_{j,\tau}^{(k-1)} =0.
\end{equation}
These transport equations can be solved uniquely, by imposing zero initial conditions on the hyperplane
$$
\Sigma_j = \{(t,x) \in \R\times\R^{n} \,|\, 
t-s_j+(x-x_j)\cdot\xi_j= 0 \}.
$$
This yields
    \begin{align}\label{def_ak}
v_{j,\tau}^{(k)}(s+\tilde{\tau},s\xi_j+y)=\frac{1}{2{\rm i}}\int_0^s ((\Box+V)v_{j,\tau}^{(k-1)})(\tilde{s}+\tilde{\tau},\tilde{s}\xi_j+y)\,d\tilde{s},
    \end{align}
where $s \in \R$ and $(\tilde{\tau},y) \in(\R\times\R^n)\cap \Sigma_{j}$. It follows from (\ref{a_0}), via an induction, that also the subsequent amplitude terms are supported in a $\frac{\delta_j}{2\log\tau}$ tubular neighborhood of $\gamma_{j}$. 
\begin{remark}
We emphasize that while the principal amplitude $v_{j,\tau}^{(0)}$ does not depend on $V$, the subsequent terms $v_{j,\tau}^{(1)}$ and $v_{j,\tau}^{(2)}$ involve $V$ and its derivatives. In particular $v_{j,\tau}^{(1)}$ depends on $V$ while $v_{j,\tau}^{(2)}$ depends on $V$ and its first and second order derivatives. 
\end{remark}
\noindent We have the following bounds that follow directly from the expressions \eqref{a_0}--\eqref{def_ak}:
\begin{equation}
\label{amp_bounds}
\begin{aligned}
\|v^{(k)}_{j,\tau}\|_{\mathcal C^\ell((0,T)\times \tilde\Omega)} &\leq \kappa_{0,j}\, \,(\log\tau)^{\frac{n}{2}+2k+\ell} \quad \text{for $k=0,1,2$ and $\ell=0,1,2$}\\
\|v^{(k)}_{j,\tau}\|_{H^\ell((0,T)\times \tilde\Omega)} &\leq \kappa_{0,j} \,(\log\tau)^{2k+\ell} \quad \text{for $k=0,1,2$ and $\ell=0,1,2$}
\end{aligned}
\end{equation}
where $\kappa_{0,j}$ is a positive constant that is independent of the parameter $\tau$. Next, we use the definition of the combined amplitude term $v_{j,\tau}$ and the bounds above to deduce that
\begin{equation}
\label{amp_bound}
\begin{aligned}
\|v_{j,\tau}\|_{\mathcal C^{k}((0,T)\times\tilde\Omega)}&\leq \kappa_{1,j}\,(\log\tau)^{\frac{n}{2}+k} \quad \text{for $k=0,1,2$},\\
\|v_{j,\tau}\|_{H^{k}((0,T)\times\tilde\Omega)}&\leq \kappa_{1,j}\,(\log\tau)^{k} \quad \text{for $k=0,1,2$}
\end{aligned}
\end{equation}
where $\kappa_{1,j}$ is a positive constant that is independent of the parameter $\tau$. Similarly using equation \eqref{ansatz} together with the latter bound we deduce that
   \begin{equation}
\label{principal_bound}
 (\log\tau)^{\frac{n}{2}}\|\mathcal U_{j,\tau}\|_{H^{k}((0,T)\times\tilde{\Omega})}+ \|\mathcal U_{j,\tau}\|_{\mathcal C^{k}((0,T)\times\tilde{\Omega})}\leq \kappa_{2,j}\,\tau^{k}(\log\tau)^{\frac{n}{2}}\quad \text{for $k=0,1,2,$}
 \end{equation}
where $\kappa_{2,j}$ is a positive constant that is independent of the parameter $\tau$. Moreover, equations (\ref{transp_a0}) and (\ref{transmin}), together with (\ref{conjugatego}) imply that
$$ (\Box+V)\,\mathcal U_{j,\tau}= \tau^{-2}\,e^{{\rm i} \tau(-t+\xi_j\cdot x)}(\Box+V)v_{j,\tau}^{(2)}$$
and therefore
    \begin{align}\label{go_remainder}
\|(\Box+V)\,\mathcal U_{j,\tau}\|_{H^{1}((0,T) \times \tilde\Omega)} \leq \kappa_{3,j}\, \tau^{-1}\,(\log\tau)^{6},
    \end{align}
where $\kappa_{3,j}$ is a positive constant that is independent of the parameter $\tau$. 

Let us now consider the source term $f_{j,\tau}$ defined through the expression
\begin{equation}
\label{packets}
f_{j,\tau}(t,x) = \zeta_{j,+}(t)\,\Box(\zeta_{j,-}(t)\,\mathcal U_{j,\tau}(t,x)),\quad (t,x)\in\R^{1+n}
\end{equation}
where $\zeta_{j,\pm}$ are as defined in Section~\ref{prelim_sec}. From the definition of $\zeta_{j,\pm}$, we deduce that
\begin{equation}\label{sup1}f_{j,\tau}(t,x)=0,\quad t\notin\left[s_j-\frac{\delta_j}{4\sqrt{n}},s_j+\frac{\delta_j}{4\sqrt{n}}\right],\ x\in\R^n.\end{equation}
Then,  from the condition imposed to the cut-off function $\chi$, we get
$$f_{j,\tau}(t,x)=0,\quad t\in\left(s_j-\frac{\delta_j}{4\sqrt{n}},s_j+\frac{\delta_j}{4\sqrt{n}}\right),\ |s_j-t+(x-x_j)\cdot\xi_j|\geq \frac{\delta_j}{4\sqrt{n}}$$
which implies that
\begin{equation}\label{sup2}f_{j,\tau}(t,x)=0,\quad t\in\left(s_j-\frac{\delta_j}{4\sqrt{n}},s_j+\frac{\delta_j}{4\sqrt{n}}\right),\ |(x-x_j)\cdot\xi_j|\geq \frac{\delta_j}{2\sqrt{n}}.\end{equation}
In the same way, for all $k=1,\ldots,n-1$, using the fact that
$$\chi[(\log\tau)\,\delta_j^{-1}(x-x_j)\cdot e_{j,k} ]=0,\quad |(x-x_j)\cdot e_{j,k}|\geq \frac{\delta_j}{2\sqrt{n}},$$
we obtain 
\begin{equation}\label{sup3}f_{j,\tau}(t,x)=0,\quad t\in\R,\ |(x-x_j)\cdot e_{j,k}|\geq \frac{\delta_j}{2\sqrt{n}}.\end{equation}
Combining \eqref{sup1}--\eqref{sup3} with the fact that $\{\xi_j,e_{j,1},\ldots,e_{j,n-1}\}$ form an orthonormal basis for $\R^n$, we deduce that
\begin{equation}\label{sup4} f_{j,\tau}(t,x)=0,\quad t\in\R,\ |x-x_j|\geq \frac{\delta_j}{2}.\end{equation}
Here we use the fact that the condition $|x-x_j|\geq \frac{\delta_j}{2}$ implies that either
$$|(x-x_j)\cdot\xi_j|\geq \frac{\delta_j}{2\sqrt{n}}$$ or there exists $k\in\{1,\cdots,n-1\}$ such that
$|(x-x_j)\cdot e_{j,k}|\geq \frac{\delta_j}{2\sqrt{n}}.$
The identity \eqref{sup1} and \eqref{sup4} imply that $$\supp (f_{j,\tau}) \subset B_{\delta_j}(q_j)\subset (0,T)\times (\tilde{\Omega}\setminus\overline{\Omega}).$$
\begin{remark}
We emphasize that the source function $f_{j,\tau}$ is explicitly known, independent of the potential $V$, since it is supported in $B_{\delta_j}(q_j)$ and the function $\mathcal U_{j,\tau}$ is explicitly known here since its construction is local around $q_j$ and $V$ vanishes there. 
\end{remark}
\noindent We also record that 
\begin{equation}
\label{f_bound0}
\|f_{j,\tau}\|_{H^k((0,T)\times\tilde{\Omega})}\leq \kappa_{4,j}\,\tau^{1+k} \quad \text{for}\quad k=0,1.
\end{equation}
where $\kappa_{4,j}$ is a positive constant that is independent of the parameter $\tau$.

Next we define $u_{j,\tau}$ as the unique solution to equation \eqref{pf} subject to the source function $f_{j,\tau}$. Recalling the fact that $V$ vanishes in a $\delta_j$ neighborhood of $q_j$, we write
$$\begin{aligned}(\Box + V)(u_{j,\tau}-\zeta_{j,-}(t)\,\mathcal U_{j,\tau}) &= f_{j,\tau}- (\Box+V)(\zeta_{j,-}(t)\, \mathcal U_{j,\tau})\\
\ &=(\zeta_{j,+}(t)-1)\,(\Box+V)\,\mathcal U_{j,\tau},\end{aligned}$$
where we used the fact that $\zeta_{j,-}=1$ on a neighborhood of the support of $1-\zeta_{j,+}$.
Writing 
\begin{equation}
\label{correction_term}
 u_{j,\tau}(t,x)= \zeta_{j,-}(t)\,\mathcal U_{j,\tau}(t,x)+R_{j,\tau}(t,x),\quad (t,x)\in (0,T)\times \R^{n}
 \end{equation}
and applying the bound \eqref{go_remainder} together with classical energy estimates for the wave equation, we deduce the following bounds for the correction term $R_{j,\tau}$:
\begin{equation}
\label{correction_bound}
\|R_{j,\tau}\|_{H^2((0,T)\times \tilde\Omega)} \leq \kappa_{5,j}\,\tau^{-1}(\log\tau)^6
\end{equation}
where $\kappa_{5,j}$ is a positive constant that is independent of the parameter $\tau$. Let us also recall from Section~\ref{ray_sec} that $\zeta_{j,-}=1$ on the segment of the light ray $\gamma_j$ that lies inside $(0,T)\times \Omega$. Thus, the source term $f_{j,\tau}$ generates a solution that is approximately equal to the geometric optic ansatz $\mathcal U_{j,\tau}$ on $(0,T)\times \Omega$. 

Finally and for the sake of brevity, we define for each $j \in \N$, the positive constant $\kappa_j$ through the expression
\begin{equation}
\label{kappa}
\kappa_j =\delta_j^{-2}\cdot\max{\{\kappa_{0,j},\ldots,\kappa_{5,j}\}}.
\end{equation}
We mention in passing that $\kappa_j$ can for example be chosen to be $C\,\delta_j^{-\frac{n}{2}-8}$ where $C$ is a sufficiently large constant depending only on $T$, $\tilde\Omega$ and an a priori bound on $\|V\|_{C^4((0,T)\times\Omega)}$.
\section{Construction of the universal source function}
\label{universal_source_sec}

Let $\tau_k=e^k$ for $k\in\mathbb N$ and define the sequence $\{c_k\}_{k=1}^{\infty}$ through
$$ c_k =k^{-3}\tau_k^{-1},\quad k\in\mathbb N$$
%Noting that for each $i\in \N$
%$$ \sum_{j=1}^{\infty} c_j\, \|f_{i,\tau}\|_{C((0,T)\times \tilde \Omega)} <\kappa_i \sum_{j=1}^{\infty} c_j \,\tau_j < \frac{\pi^2}{6}\kappa_i$$
We proceed to define for each $j \in \N$, a source term $f_{j}\in L^2((0,T)\times (\tilde{\Omega}\setminus\overline{\Omega}))$ through the expression 
$$ f_{j} = \lim_{N \to \infty}\sum_{k=1}^{N} c_k\, f_{j,\tau_k} \quad \text{in $L^2((0,T)\times \tilde \Omega)$ topology.}$$
Here, the sources $f_{j,\tau_k}$ are given by expression \eqref{packets}. Observe that this definition is justified since by \eqref{f_bound0} we have
$$ \sum_{k=1}^{\infty}c_k\,\|f_{j,\tau_k}\|_{L^2((0,T)\times\tilde \Omega))} \leq \kappa_{j} \sum_{k=1}^{\infty} c_k\, \tau_k\leq C\kappa_{j},$$
with $C>0$ independent of $j$.
Since all the sources $f_{j,\tau_k}$ are supported in balls of radius $\delta_j$ centered at points $q_{j}$ we also have
$$\supp f_{j} \subset (0,T)\times (\tilde{\Omega}\setminus\overline{\Omega}).$$
Henceforth, we will use the formal notation 
$$f_{j}= \sum_{k=1}^{\infty} c_k \,f_{j,\tau_k}$$
noting that the convergence is implicitly implied in the $L^2((0,T)\times \tilde\Omega)$ topology. Next, we define a sequence of positive real numbers $\{b_j\}_{j=1}^{\infty}$ such that 
\begin{equation}
\label{b_sequence}
\sum_{j=1}^{\infty} b_j \,\kappa_j <\infty.
\end{equation}
We now define our universal source function through the expression 
\begin{equation}
\label{universal_source}
f = \lim_{N \to \infty}\sum_{j=1}^{N} b_j\, f_{j} \quad \text{in $L^2((0,T)\times \tilde \Omega)$ topology.}
\end{equation}
Observe that $f \in L^2((0,T)\times\tilde{\Omega})$ and $\supp f\subset (0,T)\times (\tilde{\Omega}\setminus\overline{\Omega})$.

With the construction of the universal source function $f$ completed as above, we proceed to study \eqref{pf} subject to this source term. Let $u_{j,\tau_k}$ denote the solution to \eqref{pf} subject to the source $f_{j,\tau_k}$. Applying the energy estimate \eqref{energy}, it follows that
$$\sum_{k=1}^{\infty} c_k \,\|u_{j,\tau_k}\|_{\mathcal C^1([0,T];L^2(\tilde \Omega))\cap\mathcal C([0,T];H^1(\tilde \Omega))} \leq\,C\,\sum_{k=1}^{\infty}\|f_{j,\tau_k}\|_{L^2((0,T)\times\tilde \Omega)}\leq C\kappa_j.$$
Therefore
\begin{equation}\label{conv_u} 
\sum_{j=1}^{\infty}b_j\,\sum_{k=1}^{\infty} c_k \,\|u_{j,\tau_k}\|_{\mathcal C^1([0,T];L^2(\tilde \Omega))\cap\mathcal C([0,T];H^1(\tilde \Omega))} \leq C\sum_{j=1}^{\infty}b_j\,\kappa_j<\infty
\end{equation}
where we used \eqref{b_sequence} in the last step. Thus, we can define the function
\begin{equation}
\label{u_exp}
 u = \sum_{j=1}^{\infty}b_j\,( \underbrace{\sum_{k=1}^{\infty} c_k\,u_{j,\tau_k}}_{u_j})
\end{equation}
where the convergence of the infinite series holds with respect to the $$\mathcal C^1([0,T];L^2(\tilde \Omega))\cap\mathcal C([0,T];H^1(\tilde \Omega))$$ topology. Applying Lemma~\ref{conv_lem}, we conclude that the function $u$ above is the unique solution to \eqref{pf} subject to the universal source function $f$ given by \eqref{universal_source}.

\section{A representation formula}
\label{asymptotic_analysis_sec}

Let us consider a fixed $j \in \N$ corresponding to a fixed $\gamma_j \in \mathbb V$ and define 
\begin{equation}
\label{I_N}
I_{N}^{j}= \int_0^T\int_{\tilde{\Omega}\setminus\overline{\Omega}} \left(f\, \eta_j(x)\,\mathcal W_{j,\tau_N}-(\Box(\eta_j(x) \mathcal W_{j,\tau_N})-\eta_j(x) \,\Box \mathcal W_{j,\tau_N})\,u\right) \,dxdt,
\end{equation}
where $u$ solves \eqref{pf}.
Here,
$$ \mathcal W_{j,\tau_N}(t,x)=  e^{-{\rm i} \tau_N(-t+\xi_j\cdot x)}\, w_{j,N}(t,x)$$
with
$$ w_{j,N} (t,x)=\left(\frac{N}{\delta_j}\right)^{\frac{n}{2}}\,\chi[N\,{\delta_j}^{-1}(s_j-t+\xi_j\cdot (x-x_j))] \prod_{k=1}^{n-1}\chi(N\,\delta_j^{-1}e_{j,k}\cdot (x-x_j))$$
and $\eta_j\in C^{\infty}_c(\tilde{\Omega})$ is chosen such that $\eta_j \equiv 1$ on $\Omega$ and $\eta_j(x)=0$ for all $x\in \tilde{\Omega}\setminus\overline{\Omega}$ such that $\dist{(x',\p\Omega)}>\frac{\delta_j}{4}$. We also require that
$$ \|\eta_j \|_{C^2(\tilde\Omega)} \leq C\,\delta_j^{-2}$$
for some constant $C>0$ independent of $j$.

Let us emphasize that the dependency of $I_N^{j}$ with respect to the coefficient $V$ is given by  $u_{|(0,T)\times\tilde{\Omega}\setminus\overline{\Omega}}$ with $u$ the solution of \eqref{pf}. Therefore, if $u_{|(0,T)\times\tilde{\Omega}\setminus\overline{\Omega}}$ is known, $I_N^{j}$ will be also known even if the coefficient $V$ is unknown.   The definition of $I_N^{j}$ is motivated by the following computation:
\begin{align*}
\int_0^T\int_{\tilde{\Omega}\setminus\overline{\Omega}}f\, \eta_j(x)\,\mathcal W_{j,\tau_N}\,dx=& \int_0^T\int_{\tilde{\Omega}} \left((\Box+V)u\right)\eta_j(x)\,\mathcal W_{j,\tau_N}\,dx\,dt\\
=&\int_{(0,T)\times \tilde{\Omega}}u\,\Box(\eta_j(x)\,\mathcal W_{j,\tau_N})\,dx\,dt\\
&+\int_{(0,T)\times \tilde{\Omega}} V\,u\,\eta_j(x)\,\mathcal W_{j,\tau_N}\,dx\,dt, 
\end{align*}
where $u$ solves \eqref{pf} and we have used integration by parts in the second step. There are no boundary terms on $(0,T)\times\p\tilde\Omega$ since $\eta_j$ vanishes there. Moreover, no boundary terms appear at $t=0$ or $t=T$ since $u,\p_t u$ vanish at $t=0$ while $(t,x)\mapsto\eta_j(x)\,\mathcal W_{j,\tau_N}(t,x)$ is supported away from $\{T\}\times\tilde\Omega$. This implies that
\begin{equation}
\label{I_exp}
I_N^{j}=\int_{(0,T)\times\tilde\Omega} e^{-{\rm i}\tau_N(-t+\xi_j\cdot x)}\, u\, \underbrace{\eta_j\,(\Box+V) w_{j,N}}_{\tilde{w}_{j,N}}\,dx\,dt
\end{equation}
Let us record in passing that the function $\tilde{w}_{j,N}$ is compactly supported in $(0,T)\times \tilde{\Omega}$ and that 
\begin{equation}
\label{w_bound}
\|w_{j,N}\|_{H^k((0,T)\times \tilde{\Omega})}+N^{-2}\|\tilde{w}_{j,N}\|_{H^k((0,T)\times \tilde{\Omega})}< C \kappa_j \, N^{k}\quad\text{for}\quad \, k=0,1,2
\end{equation}
for some $C>0$ independent of $j$ and $N$, where we recall that $\kappa_j$ is as defined in \eqref{kappa}.
For the remainder of this section, we aim to prove the following theorem.
%\bigskip
\begin{theorem}
\label{t1}
Let $f$ be the universal source given by \eqref{universal_source} and let $$V \in \mathcal C^4([0,T]\times\Omega)\cap \mathcal C([0,T];\mathcal C^4_0(\Omega)).$$ 
For each $j \in \N$, there holds: 
\begin{equation}
\label{t1a} \lim_{N \to \infty} \left(c_N^{-1}\,I_N^{j}-S^{j}_N\right) =\frac{\sqrt{2}}{2}\,b_j\,\int_{\R} V(\gamma_j(s))\,ds.\end{equation}
where $I_N^j$ is as defined in \eqref{I_N} and 
$$ S_N^{j} =\sum_{k=1}^{\infty}b_{k} \int_{(0,T)\times \tilde{\Omega}}e^{{\rm i}\tau_N(\xi_k-\xi_j)\cdot x}\, \eta_j(x)\,\zeta_{k,-}(t)\,v_{k,\tau_N}^{(0)}(t,x)\,\Box w_{j,N}(t,x)\,dx\,dt$$
is an explicit constant depending only on $N$, $j$, $T$, $\tilde\Omega$ and $\Omega$.
\end{theorem}
We remark that for each fixed $j$ and $N$, the expression for $S_N^j$ is well-defined by \eqref{b_sequence} and \eqref{w_bound}. Let us make a preliminary computation to divide the analysis of the limit in Theorem~\ref{t1} into two components. To this end we use \eqref{u_exp} to write
\[
\begin{aligned}
I_N^{j}&=\int_{(0,T)\times\tilde\Omega} \left(\sum_{k=1}^{\infty}b_k\, e^{-{\rm i}\tau_N(-t+\xi_j\cdot x)}\, u_k\,\tilde{w}_{j,N}\right)\,dx\,dt\\
&=\sum_{k=1}^{\infty}b_k\,\int_{(0,T)\times\tilde\Omega} e^{-{\rm i}\tau_N(-t+\xi_j\cdot x)}\, u_{k}\,\tilde{w}_{j,N}\,dx\,dt\\
&= \sum_{k=1}^{\infty} b_k\,\int_{(0,T)\times \tilde\Omega} e^{-{\rm i}\tau_N(-t+\xi_j\cdot x)}\,\left(\sum_{\ell=1}^{\infty} c_\ell u_{k,\tau_\ell}\right)\,\tilde{w}_{j,N}\,dx\,dt\\
&= \sum_{k=1}^{\infty} b_k\,\sum_{\ell=1}^{\infty}c_\ell\,\int_{(0,T)\times \tilde\Omega} e^{-{\rm i}\tau_N(-t+\xi_j\cdot x)}\,u_{k,\tau_\ell}\,\tilde{w}_{j,N}\,dx\,dt
\end{aligned}
\]
The interchanging of the integration and the limits are justified by \eqref{conv_u}. The latter expression can be rewritten as

$$ I_N^{j}=b_j\,J_N^{j}+K_N^{j}$$ 
where
\begin{equation}
\label{JK_def}
 \begin{aligned}
 J_N^{j}&=\sum_{\ell=1}^{\infty}c_\ell\,\int_{(0,T)\times \tilde\Omega} e^{-{\rm i}\tau_N(-t+\xi_j\cdot x)}\,u_{j,\tau_\ell}\,\tilde{w}_{j,N}\,dx\,dt \\
 K_N^{j}&= \sum_{k\neq j}b_{k}\, \sum_{\ell=1}^{\infty}c_\ell\,\int_{(0,T)\times \tilde\Omega} e^{-{\rm i}\tau_N(-t+\xi_j\cdot x)}\,u_{k,\tau_\ell}\,\tilde{w}_{j,N}\,dx\,dt.
\end{aligned}
\end{equation}

We proceed to study asymptotic behavior of these two terms as $N$ approaches infinity.

\begin{remark}
\label{C_generic}
In what follows, we will use the symbol $C$ to denote a generic positive constant that is independent of the indices $j$ and $N$ in $I^j_N$ and that only depends on $T$, $\tilde\Omega$, $\Omega$ and $\|V\|_{C^4((0,T)\times\tilde\Omega)}$.
\end{remark}

\subsection{Asymptotic analysis of $J_N^{j}$}

The aim of this section is to prove the following lemma.

\begin{lemma}
\label{lem_J_bound}
Let $j\in\N$ and $J_N^{j}$ be defined through \eqref{JK_def}. Then
$$\lim_{N\to \infty} \left|c_N^{-1}\,J_N^{j}- \int_{(0,T)\times \tilde \Omega}\zeta_{j,-}(t)\,\tilde{w}_{j,N}\,v^{(0)}_{j,\tau_N}\,dx\, dt\right|=0.$$
\end{lemma}

Applying the definition \eqref{correction_term}, we split the expression for $J_N^{j}$ into two terms $J_N^{j}=J_{1,N}^{j}+J_{2,N}^{j}$ where
$$J_{1,N}^{j}=\sum_{\ell=1}^{\infty}c_\ell\,\int_{(0,T)\times \tilde\Omega} e^{-{\rm i}\tau_N(-t+\xi_j\cdot x)}\,\zeta_{j,-}(t)\,\mathcal U_{j,\tau_\ell}\,\tilde{w}_{j,N}\,dx\,dt$$
and
$$J_{2,N}^{j}=\sum_{\ell=1}^{\infty}c_\ell\,\int_{(0,T)\times \tilde\Omega} e^{-{\rm i}\tau_N(-t+\xi_j\cdot x)}\,R_{j,\tau_\ell}\,\tilde{w}_{j,N}\,dx\,dt.$$
Note that this breaking of the infinite series $J_N^{j}$ is justified again since each of the series $J_{1,N}^{j}$ and $J_{2,N}^{j}$ are absolutely convergent. 
\subsubsection{Asymptotic analysis of $J_{1,N}^{j}$}
\label{subsection_1}
Observe that 
\begin{multline*}
J_{1,N}^{j}=c_N  \int_{(0,T)\times \tilde \Omega}\zeta_{j,-}\,\tilde{w}_{j,N}\,v_{j,\tau_N}\,dx\,dt\\
+\sum_{\ell\neq N}c_\ell\,(\tau_\ell-\tau_N)^{-2} \int_{(0,T)\times \tilde \Omega} \p^2_t\left(\zeta_{j,-}\,\tilde{w}_{j,N}\,v_{j,\tau_\ell}\right)\, e^{{\rm i}(\tau_\ell-\tau_N)(-t+\xi_j\cdot x)}\,dx\,dt,
\end{multline*}
where we have isolated the summation index $\ell=N$ and performed integration by parts with respect to the time variable twice in the summation over indices $\ell \neq N$, also using the fact that the function $\tilde{w}_{j,N}$ is compactly supported in the set $(0,T)\times \tilde\Omega$. Next, recalling the definition \eqref{kappa} together with the estimates \eqref{amp_bound}, \eqref{w_bound} and $$\|\zeta_{j,\pm}\|_{\mathcal C^2((0,T)\times\tilde\Omega)}\leq C\delta_j^{-2},$$
we obtain that
\begin{equation}
\label{alt_w_bound}
\|\zeta_{j,-}\,\tilde{w}_{j,N}\,v_{j,\tau_\ell}\|_{H^2((0,T)\times\tilde\Omega)}< C\kappa_j^2 \,N^{4}\,\ell^2.
\end{equation}
Combining this with the bound
\begin{equation}
\label{tau_diff}
|\tau_\ell-\tau_N|=|e^\ell-e^N|\geq |e^{N-1}-e^N|=\frac{e-1}{e}e^N\geq \frac{\tau_N}{2}\quad \quad \ell\neq N,
\end{equation}
we deduce that 
\begin{equation}
\label{J1_bound}\begin{aligned}
\left|c_N^{-1}J_{1,N}^{j}- \int_{(0,T)\times \tilde \Omega}\zeta_{j,-}\,\tilde{w}_{j,N}\,v_{j,\tau_N}\,dx\,dt\right| &\leq 4C\kappa_j^2 N^{4}c_N^{-1} \tau^{-2}_N (\sum_{\ell\neq N} c_\ell\,\ell^2) \\
&\leq 4C\kappa_j^2 N^{7} \tau^{-1}_N,\end{aligned}
\end{equation}
where we recall the notation from Remark~\ref{C_generic} that $C>0$ is a constant independent of $j$ and $N$.
\subsubsection{Asymptotic analysis of $J_{2,N}^{j}$}
\label{subsection_2}
We write
\begin{equation}
\label{J2_bound}
\begin{aligned}
|J_{2,N}^{j}|&=\left| \sum_{\ell=1}^{\infty}c_\ell\,\int_{(0,T)\times \tilde\Omega} e^{-{\rm i}\tau_N(-t+\xi_j\cdot x)}\,R_{j,\tau_\ell}\,\tilde{w}_{j,N}\,dx\,dt \right|\\
&=\left|\sum_{\ell=1}^{\infty} c_\ell\,\tau_N^{-2}\int_{(0,T)\times\tilde\Omega} \p^2_t\left(\tilde{w}_{j,N}\,R_{j,\tau_\ell}\right)e^{-{\rm i}\tau_N(-t+\xi_j\cdot x)}\,dx\,dt\right|\\
&<\sum_{\ell=1}^{\infty}c_\ell\,\tau_N^{-2} \|R_{j,\tau_\ell}\|_{H^2((0,T)\times\tilde\Omega)}\|\tilde{w}_{j,N}\|_{H^2((0,T)\times\tilde\Omega)}\\
&< \kappa_j^2\,\tau_N^{-2}N^{4} (\sum_{\ell=1}^{\infty} c_\ell\,\tau_\ell^{-1}\,\ell^6)<C\kappa_j^2\,\tau_N^{-2}\,N^{4}
\end{aligned}
\end{equation}
where we have integrated by parts in time twice and used the bounds \eqref{correction_bound} and \eqref{w_bound}. Combining the bounds given by \eqref{J1_bound} and \eqref{J2_bound}, together with the fact that
\begin{equation}
\label{limit_iden}
\lim_{N\to\infty}c_N^{-1}\tau_N^{-2}N^{4}=\lim_{N\to\infty}\tau_N^{-1}N^{7}=0
\end{equation}

\begin{proof}[Proof of Lemma~\ref{lem_J_bound}]
Note that by combining the estimates for $J_{1,N}^j$ and $J^j_{2,N}$ we have shown that
\begin{equation}\label{lem_J_almost}\lim_{N\to \infty} \left|c_N^{-1}\,J_N^{j}- \int_{(0,T)\times \tilde \Omega}\zeta_{j,-}(t)\,\tilde{w}_{j,N}\,v_{j,\tau_N}\,dx\, dt\right|=0.\end{equation}
Moreover, using \eqref{ansatz} and \eqref{amp_bounds} we obtain
\begin{equation}
\label{a_approx}
\|v_{j,\tau_N}-v_{j,\tau_N}^{(0)}\|_{L^2((0,T)\times\Omega)}\leq C\kappa_j\,N^2\,\tau_N^{-1} \quad \forall \, j\in \N.
\end{equation}
This bound implies that
$$\begin{aligned}&\lim_{N\to\infty} \left|\int_{(0,T)\times\tilde\Omega}\zeta_{j,-}\,\tilde{w}_{j,N}\,v_{j,\tau_N}^{(0)}\,dx\,dt-\int_{(0,T)\times\tilde\Omega}\zeta_{j,-}\,\tilde{w}_{j,N}v_{j,\tau_N}\,dx\,dt\right|\\
&=\lim_{N\to\infty}N^2\tau_N^{-1}=0.\end{aligned}$$
The claim now follows immediately from this estimate and \eqref{lem_J_almost}.
\end{proof}
\subsection{Asymptotic analysis of $K_N^{j}$} 
In this section we prove the following lemma.

\begin{lemma}
\label{lem_K_bound}
Let $j\in \N$ and $K_N^{j}$ be defined through \eqref{JK_def}. Then
$$\lim_{N\to \infty} \left|c_N^{-1}\,K_N^{j}-\sum_{k\neq j}b_{k} \int_{(0,T)\times \tilde{\Omega}}e^{{\rm i}\tau_N  (\xi_k-\xi_j)\cdot x}\, \eta_j\,\zeta_{k,-}\,v^{(0)}_{k,\tau_N}\,\Box w_{j,N}\,dx\,dt\right|=0.$$
\end{lemma}
Applying the definition \eqref{correction_term}, we can split the expression for $K_N^{j}$ into three terms $$K_N^j=K_{1,N}^{j}+K_{2,N}^{j}+K_{3,N}^{j}$$ with
$$K_{1,N}^{j}=\sum_{k \neq j}b_{k}\,c_N\,\int_{(0,T)\times \tilde\Omega} e^{{\rm i}\tau_N(\xi_k-\xi_j)\cdot x}\,\zeta_{k,-}(t)\,v_{k,\tau_N}\,\tilde{w}_{j,N}\,dx\,dt,$$
$$K_{2,N}^{j}=\sum_{k \neq j}b_{k}\,\sum_{\ell\neq N}c_\ell\,\int_{(0,T)\times \tilde\Omega} e^{-{\rm i}\tau_N(-t+\xi_j\cdot x)}\,\zeta_{k,-}\,\mathcal U_{k,\tau_\ell}\,\tilde{w}_{j,N}\,dx\,dt,$$
$$K_{3,N}^{j}=\sum_{k\neq j}b_{k}\,\sum_{\ell=1}^{\infty}c_\ell\,\int_{(0,T)\times \tilde\Omega} e^{-{\rm i}\tau_N(-t+\xi_j\cdot x)}\,R_{k,\tau_\ell}\,\tilde{w}_{j,N}\,dx\,dt.$$

We emphasize that this step is justified since all three series converge absolutely. We proceed to bound each of the three terms above.
\subsubsection{Asymptotic analysis of $K_{1,N}^{j}$}
We show in this section that 
\begin{equation}
\label{K1Nbound}
\lim_{N \to\infty} \left|c_N^{-1} K_{1,N}^{j}- \sum_{k\neq j}b_{k} \int_{(0,T)\times \tilde{\Omega}}e^{{\rm i}\tau_N(\xi_k-\xi_j)\cdot x}\,\zeta_{k,-}(t)\,v_{k,\tau_N}\,\Box w_{j,N}\,dx\,dt \right| =0.
\end{equation}
Note that it suffices to show that 
\begin{equation}
\label{Kbound_rand}
\lim_{N\to \infty}\left| \sum_{k \neq j}b_{k}\,\int_{(0,T)\times \tilde\Omega}e^{{\rm i}\tau_N(\xi_k-\xi_j)\cdot x}\,\zeta_{k,-}(t)\,v_{k,\tau_N}\,\eta_j\,V\,w_{j,N}\,dx\,dt   \right|=0.
\end{equation}
Before proving this limit, we need to make a definition. For each $j,k\in \N$, we set
$$\theta_{k,j}:=\inf_{s,\tilde{s}\in\R}\dist{(p_{l_k}+s\xi_k,p_{l_j}+\tilde{s}\xi_j)},$$
where we recall that                                                              
$$\gamma_j(s)=(t_{k_j}+s,p_{l_j}+s\xi_j),\quad s\in\R.$$
Then, for all $j\in\N$, we define the function $h_{j}:\N\to \N$ through
$$ h_{j}(r) = \min \{k \in \N\setminus \{j\}\,:\,|\xi_k-\xi_j|<\tau_r^{-\frac{1}{2}},\quad \theta_{k,j}<(\delta_j+\delta_k)\,r^{-\frac{1}{2}}\},$$
This minimum always exists since $\mathcal T \times \mathcal P$ is dense in $(0,T)\times \p\Omega$ and the sequence $\{\delta_k\}_{k\in\N}$ is a decreasing sequence. We claim that
\begin{equation}
\label{density}
\lim_{r\to\infty} h_j(r) = \infty,\quad\quad  j \in \N.
\end{equation}
To show this, we suppose for contrary that there exists an integer $j$, a strictly increasing sequence $\{r_k\}_{k=1}^{\infty}$ and an integer $N_0$, such that $h_j(r_k)\leq N_0$ for all $k\in\N$. Note first that
$$\limsup_{k\to\infty}|\xi_{h_j(r_k)}-\xi_j|\leq \limsup_{k\to\infty}\tau_{r_k}^{-\frac{1}{2}}=0,$$
$$\limsup_{k\to\infty}\theta_{h_j(r_k),j}<\limsup_{k\to\infty}(\delta_j+\delta_{h_j(r_k)})\,r_k^{-\frac{1}{2}}\}=0.$$
Combining this with the fact that  the set $\{1,\ldots,N_0\}$ is finite, we deduce that there  exists an index $k_0$ such that for $k=h_j(r_{k_0})$ we have
$$ \xi_k=\xi_j\quad \text{and}\quad\inf_{s,\tilde{s}\in\R}\dist{(p_{l_k}+s\xi_k,p_{l_j}+\tilde{s}\xi_j)}=0.$$
But then $h_j(r_{k_0})=j$ which contradicts the definition of $h_j$. Thus, \eqref{density} holds.

We return to the expression \eqref{Kbound_rand} and rewrite it as 
\begin{multline}
\label{Kbound_rand_1}
\sum_{k< h_j(N),\,k\neq j} b_{k}\, \int_{(0,T)\times \tilde \Omega}\zeta_{k,-}\,\eta_j\,V\,w_{j,N}\,v_{k,\tau_N}\, e^{{\rm i} \tau_N(\xi_k-\xi_j)\cdot x}\,dx\,dt\\
+\sum_{k\geq h_j(N),\,k\neq j}b_{k}\, \int_{(0,T)\times \tilde \Omega}\zeta_{k,-}\,\eta_j\,V\,w_{j,N}\,v_{k,\tau_N}\, e^{{\rm i} \tau_N(\xi_k-\xi_j)\cdot x}\,dx\,dt.
\end{multline}
Let us begin by analyzing the first term in the expression \eqref{Kbound_rand_1}. We note that, from the definition of the map $h_j$, given any $k < h_j(N)$ with $k\neq j$, either \begin{equation}\label{tt1}|\xi_k-\xi_j|\geq \tau_N^{-\frac{1}{2}}\end{equation} or 
\begin{equation}\label{tt2}\inf_{s,\tilde{s}\in\R}\dist{(p_{l_k}+s\xi_k,p_{l_j}+\tilde{s}\xi_j)}\geq (\delta_j+\delta_{k})\,N^{-\frac{1}{2}}\end{equation} 
holds true. In the latter scenario, the terms in the summation vanish. To see this, note that the terms $w_{j,N}$ and $v_{k,\tau_N}$ are supported in tubular neighborhoods of $\gamma_j$ and $\gamma_{k}$ of radius $\frac{\delta_j}{2N}$ and $\frac{\delta_{k}}{2N}$ respectively. Therefore, the condition \eqref{tt2} implies that
$$ v_{k,\tau_N}\,\tilde{w}_{j,N}\equiv 0,\quad  N\in \N.$$
In the former scenario, integrating by parts, we get
$$\begin{aligned}&\int_{(0,T)\times \tilde \Omega}\zeta_{k,-}(t)\,\eta_j\,V\,w_{j,N}\,v_{k,\tau_N}\, e^{{\rm i} \tau_N(\xi_k-\xi_j)\cdot x}\,dx\,dt\\
&=\int_{(0,T)\times \tilde \Omega}\zeta_{k,-}(t)\,\eta_j\,V\,w_{j,N}\,v_{k,\tau_N}\, \frac{(\xi_k-\xi_j)\cdot\nabla_xe^{{\rm i} \tau_N(\xi_k-\xi_j)\cdot x}}{{\rm i} \tau_N\left|\xi_k-\xi_j\right|^2}\,dx\,dt\\
&=\frac{{\rm i}\int_{(0,T)\times \tilde \Omega}e^{{\rm i} \tau_N(\xi_k-\xi_j)\cdot x}\,\zeta_{k,-}(t)\,(\xi_k-\xi_j)\cdot\nabla_x\left[\eta_j\,V\,w_{j,N}\,v_{k,\tau_N}\right]\,dx\,dt}{ \tau_N\left|\xi_k-\xi_j\right|^2}.\end{aligned}$$
Then, \eqref{tt1} implies
$$\begin{aligned}&\left|\int_{(0,T)\times \tilde \Omega}\zeta_{k,-}(t)\,\eta_j\,V\,w_{j,N}\,v_{k,\tau_N}\, e^{{\rm i} \tau_N(\xi_k-\xi_j)\cdot x}\,dx\,dt\right|\\
&\leq C\tau_N^{-\frac{1}{2}}\|\eta_j\,V\,w_{j,N}\,v_{k,\tau_N}\|_{L^1(0,T;W^{1,1}(\tilde \Omega))}\\
&\leq C\tau_N^{-\frac{1}{2}}(\|\eta_j\,w_{j,N}\|_{L^2(0,T;H^1(\tilde \Omega))}\|v_{k,\tau_N}\|_{L^2((0,T)\times\tilde \Omega)}\\
&\quad\quad \quad\quad\quad \quad+\|\eta_j\,w_{j,N}\|_{L^2((0,T)\times\tilde \Omega)}\|v_{k,\tau_N}\|_{L^2(0,T;H^1(\tilde \Omega))}).\end{aligned}$$
Combining this with \eqref{amp_bounds} and \eqref{w_bound}, we obtain
$$\left|\int_{(0,T)\times \tilde \Omega}\zeta_{k,-}(t)\,\eta_j\,V\,w_{j,N}\,v_{k,\tau_N}\, e^{{\rm i} \tau_N(\xi_k-\xi_j)\cdot x}\,dx\,dt\right|\leq C\kappa_j\kappa_kN\tau_N^{-\frac{1}{2}}.$$
According to the above discussion, this last estimate holds true for all $k < h_j(N)$ with $k\neq j$. Taking the sum, we deduce that
$$\begin{aligned}&\left|\sum_{k< h_j(N),\,k\neq j} b_{k}\, \int_{(0,T)\times \tilde \Omega}\zeta_{k,-}\,\eta_j\,V\,w_{j,N}\,v_{k,\tau_N}\, e^{{\rm i} \tau_N(\xi_k-\xi_j)\cdot x}\,dx\,dt\right|\\
&\leq C\kappa_jN\tau_N^{-\frac{1}{2}}\sum_{k< h_j(N),\,k\neq j} b_{k}\kappa_k\leq C\kappa_jN\tau_N^{-\frac{1}{2}}.\end{aligned}$$
Therefore, we have 
\begin{equation}\label{tt3}\lim_{N\to\infty}\sum_{k< h_j(N),\,k\neq j} b_{k}\, \int_{(0,T)\times \tilde \Omega}\zeta_{k,-}\,\eta_j\,V\,w_{j,N}\,v_{k,\tau_N}\, e^{{\rm i} \tau_N(\xi_k-\xi_j)\cdot x}\,dx\,dt=0.\end{equation}
We now consider the second term in \eqref{Kbound_rand_1}. We write
\begin{multline*}
\left|\sum_{k\geq h_j(N),\,k\neq j}b_{k}\,\int_{(0,T)\times \tilde \Omega}\zeta_{k,-}(t)\,\eta_j(x)\,V\,w_{j,N}\,v_{k,\tau_N}\, e^{{\rm i} \tau_N(\xi_k-\xi_j)\cdot x}\,dx\,dt\right|\\
\leq\sum_{k\geq h_j(N)}b_{k}\,\|v_{k,\tau_N}\,w_{j,N}\|_{L^1((0,T)\times\tilde\Omega)}\|\zeta_{k,-}\,\eta_j\,V\|_{L^{\infty}((0,T)\times\tilde\Omega)}\leq C\,\kappa_j\,\sum_{k\geq h_j(N)}\kappa_{k}b_{k}.
\end{multline*}
Here, in the last step we used the Cauchy-Schwarz inequality together with the bounds \eqref{amp_bounds} and \eqref{w_bound} to write
$$\|v_{k,\tau_N}\,w_{j,N}\|_{L^1((0,T)\times\tilde\Omega)}\leq \|v_{k,\tau_N}\|_{L^2((0,T)\times \tilde\Omega)}\|\,w_{j,N}\|_{L^2((0,T)\times \tilde\Omega)}\leq C \kappa_k\kappa_j.$$
Now, applying \eqref{b_sequence} and \eqref{density}, we conclude that 
$$ \lim_{N\to\infty} \sum_{k\geq h_j(N)}\kappa_{k}b_{k}=0.$$
Combining this with \eqref{Kbound_rand_1} and \eqref{tt3}, we deduce that \eqref{Kbound_rand} holds true. This concludes our asymptotic analysis of $K_{1,N}^{j}$ showing that \eqref{K1Nbound} is fulfilled.

\subsubsection{Asymptotic analysis of $K_{2,N}^{j}$} 
 We write 
$$ |K_{2,N}^j|=\left|\sum_{k\neq j}b_{k}\sum_{\ell \neq N}c_\ell \int_{(0,T)\times \tilde \Omega}\zeta_{k,-}\,\tilde{w}_{j,N}\,v_{k,\tau_\ell}\, e^{{\rm i} S(t,x)}\,dx\,dt\right|.$$
where we are using the shorthand notation $$S(t,x)=-(\tau_\ell-\tau_N)t+(\tau_\ell\xi_k-\tau_N\xi_j)\cdot x.$$
Using integration by parts with respect to the time variable, this reduces as follows.
$$\begin{aligned}
|K_{2,N}^j|&=\left|\sum_{k\neq j}b_{k}\sum_{\ell\neq N}c_\ell\, (\tau_\ell-\tau_N)^{-2}\int_{(0,T)\times \tilde \Omega} \p^2_t\left(\zeta_{k,-}\,\tilde{w}_{j,N}\,v_{k,\tau_\ell}\right)\, e^{{\rm i} S(t,x)}\,dx\,dt\right|\\
&\leq 4\sum_{k\neq j}b_{k}\sum_{\ell\neq N} c_\ell\, \tau_N^{-2}\| \zeta_{k,-}\,v_{k,\tau_\ell}\,\tilde{w}_{j,N}\|_{H^2((0,T)\times\tilde\Omega)}\\
&\leq\sum_{k\neq j}4\kappa_{j}\,b_{k}\kappa_{k} \tau_N^{-2} N^{4}(\sum_{\ell=1}^{\infty}c_\ell\,\ell^2)\leq C\kappa_j\,\tau_N^{-2} N^{4}.
\end{aligned}$$
where we have used estimates \eqref{alt_w_bound}--\eqref{tau_diff}. Thus, we obtain that
$$|c_N^{-1}\,K_{2,N}^{j}|\leq C\,\kappa_j\,c_N^{-1}\tau_N^{-2}N^{4}\leq C\,\kappa_j\,\tau_N^{-1}N^{7}$$
which implies that
\begin{equation}
\label{K2Nbound}
\lim_{N\to\infty}|c_N^{-1}\,K_{2,N}^{j}|=0.\end{equation}
\subsubsection{Asymptotic analysis of $K_{3,N}^{j}$}
To bound $K_{3,N}^{j}$ we write
\[
\begin{aligned}
\label{Kbound3}
&\left|\sum_{k\neq j}b_{k}\sum_{\ell=1}^{\infty}c_\ell \,\left(\int_{(0,T)\times \tilde\Omega}e^{-{\rm i}\tau_N(-t+\xi_j\cdot x)} \, R_{k,\tau_\ell}\,\tilde{w}_{j,N}\,dx\right)\right|\\
&=\left|\sum_{k\neq j}b_{k}\,\sum_{\ell=1}^{\infty} c_\ell\,\tau_N^{-2}\int_{(0,T)\times\tilde\Omega} \p^2_t\left(\tilde{w}_{j,N}\,R_{k,\tau_\ell}\right)e^{-{\rm i}\tau_N(-t+\xi_j\cdot x)}\,dx\right|\\
&\leq\sum_{k\neq j}b_{k}\,\sum_{\ell=1}^{\infty}c_\ell\,\tau_N^{-2} \|R_{k,\tau_\ell}\|_{H^2((0,T)\times\tilde\Omega)}\|\tilde{w}_{j,N}\|_{H^2((0,T)\times\tilde\Omega)}\\
&\leq \sum_{k\neq j}\kappa_j\,\kappa_{k}\,b_{k}\tau_N^{-2}N^{4} (\sum_{\ell=1}^{\infty} c_\ell\,\tau_\ell^{-1}\ell^6)\leq C\kappa_j\,\tau_N^{-2}N^{4}
\end{aligned}
\]
where we have used the bounds \eqref{correction_bound} and \eqref{w_bound}. Thus we obtain
\begin{equation}\label{K3Nbound}|c_N^{-1}K_{3,N}^{j}|\leq C\,\kappa_j\, c_N^{-1}\tau^{-2}_N\,N^{4}\leq C\,\kappa_j\,\tau_N^{-1}N^{7}.\end{equation}
We are now ready to prove Lemma~\ref{lem_K_bound} as follows.
\begin{proof}[Proof of Lemma~\ref{lem_K_bound}]
Note that by combining the estimate \eqref{K3Nbound} with \eqref{K1Nbound} and \eqref{K2Nbound}, we have shown that 
$$\lim_{N\to \infty} \left|c_N^{-1}\,K_N^{j}-\sum_{k\neq j}b_{k} \int_{(0,T)\times \tilde{\Omega}}e^{{\rm i}\tau_N  (\xi_k-\xi_j)\cdot x}\, \eta_j\,\zeta_{k,-}\,v_{k,\tau_N}\,\Box w_{j,N}\,dx\,dt\right|=0.$$
Applying the estimates \eqref{w_bound} and \eqref{a_approx}  together with the convergence of the series \eqref{b_sequence}, we observe that
\begin{multline*}\left|\sum_{k\neq j}b_{k} \int_{(0,T)\times \tilde{\Omega}}e^{{\rm i}\tau_N  (\xi_k-\xi_j)\cdot x}\, \eta_j\,\zeta_{k,-}\,(v_{k,\tau_N}-v_{k,\tau_N}^{(0)})\,\Box w_{j,N}\,dx\,dt\right| \\
\leq C\kappa_j\,N^4\tau_N^{-1} (\sum_{k=1}^{\infty}\kappa_k b_k), \end{multline*}
which converges to zero as $N$ approaches infinity. The claim follows immediately.
\end{proof}
With the proof of Lemmas \ref{lem_J_bound}--\ref{lem_K_bound} completed, we are ready to state the proof of Theorem~\ref{t1}.
\begin{proof}[Proof of Theorem~\ref{t1}]
Let $j \in \N$ corresponding to some $\gamma_j \in \mathbb V$. Combining the definition of $I^j_N$ in terms of $J^j_N$ and $K^j_N$ as given by \eqref{JK_def} together with Lemma~\ref{lem_J_bound}--\ref{lem_K_bound}, we obtain that
$$ \lim_{N\to \infty} \left|c_N^{-1}\,I_N^{j}-S^{j}_N-  b_j\,\int_{(0,T)\times \tilde \Omega}\zeta_{j,-}\,V\,\eta_{j}\,w_{j,N}v_{j,\tau_N}^{(0)}\,dx\,dt \right|=0.$$
Note that $v^{(0)}_{j,\tau_N}=w_{j,N}$. We proceed to study the expression 
\begin{equation}
\label{weak_conv} 
\int_{(0,T)\times \tilde \Omega}\eta_j\,\zeta_{j,-}\,V\,w^2_{j,N}\,dx\,dt.
\end{equation}
To simplify this expression, we introduce the new coordinate system $$(t,x^1,\ldots,x^n) \mapsto y=(y^0,\ldots,y^n)$$ on $\R^{1+n}$ that is defined by
$$(t,x^1,\ldots,x^n)= q_{j}+y^0\,\alpha_j^\star+y^1\,\alpha_{j}+\sum_{k=2}^{n}y^{k}\,e_{j,k-1}.$$
Here, $\alpha_j,\alpha_j^*\in\mathbb S^n=\{z\in\R^{1+n}:\ |z|=1\}$ are given by
$$\alpha_j=\frac{\sqrt{2}}{2}\,(-1,\xi_j),\quad \alpha_j^*=\frac{\sqrt{2}}{2}\,(1,\xi_j).$$
Note that in the $y$-coordinate system the points on the light ray $\gamma_{j}$ are given by $y^1=\ldots=y^n=0.$
Using this coordinate system together with the definitions of $v_{j,0}$ and $w_{j,N}$, the expression \eqref{weak_conv} reduces to
$$\int_{\R^{1+n}}\left(\frac{N}{\delta_j}\right)^{n}\,\eta_j(y)\,\zeta_{j,-}(y)\,V(y)\chi^2(N\,\delta_j^{-1}\sqrt{2}y^1)\left(\prod_{k=2}^{n}\chi^2(N\,\delta_j^{-1}y^k)\right)\,dy.$$
Taking the limit as $N\to \infty$ and noting that both $\eta_j$ and $\zeta_{j,-}$ are identical to one on the segment of $\gamma_j$ that lies inside  $(0,T)\times\Omega$, we obtain:
$$\begin{aligned}&\lim_{N\to\infty}\int_{\R^{1+n}}\left(\frac{N}{\delta_j}\right)^{n}\eta_j(y)\,\zeta_{i,-}(y)\,V(y)\chi^2(N\,\delta_j^{-1}\sqrt{2}y^1)\,\prod_{k=2}^{n}\chi^2(N\,\delta_j^{-1}\,y^k)\,dy\\
&= \frac{\sqrt{2}}{2}\int_\R V(y^0,0,\ldots,0)\,dy^0\end{aligned}$$
where we used \eqref{cutoff} in the last step. This completes the proof of the theorem.
\end{proof}
\section{Proofs of main results}
\label{proof_section}
This section is devoted to the proof of the main results stated in Theorem \ref{t0} and Corollary \ref{c1}. For this purpose, we will combine all the arguments of the previous sections. We start with Theorem~\ref{t0}.

\begin{proof}[Proof of Theorem~\ref{t0}] Let $f\in L^2(\R^{1+n})$ be the source term given by \eqref{universal_source} and let $V_j \in \mathcal C^4([0,T]\times\R^n)\cap \mathcal C([0,T];\mathcal C^4_0(\Omega))$, $j=1,2$. We consider also $u_j\in \mathcal C^1([0,T];L^2(\R^n)) \cap \mathcal C([0,T];H^1(\R^n))$ solving \eqref{pf} with $V=V_j$, $j=1,2$. Assuming that the condition
\begin{equation}\label{t0b}u_1(t,x)=u_2(t,x),\quad (t,x)\in(0,T)\times(\tilde{\Omega}\setminus\overline{\Omega}))\end{equation}
is fulfilled, we will prove that $V_1=V_2$ on $\mathcal D$.
We start by observing that \eqref{t0b} combined with Theorem~\ref{t1} imply
\begin{equation}\label{V_gamma} \int_{\R} (V_1-V_2)(\gamma_j(s))\,ds=0\quad \forall\, \gamma_j \in \mathbb V.
\end{equation}
Let $\gamma:\R\to\R^{1+n}$ be any future pointing light ray such that its intersection with $(0,T)\times \Omega$ lies inside $\mathcal D$. We write
$$\gamma(s) = \gamma(0)+ s\,(1,\xi)$$
for some $\gamma(0) \in (0,T)\times \p \Omega$ and some unit vector $\xi \in \R^n$. Recall that all light rays $\gamma_j \in \mathbb V$ can be written in the form
$$ \gamma_j(s)= (t_{k_j}+s,s\,\xi_j+p_{l_j})$$
for some sequences $\{k_j\}_{j=1}^{\infty}$, $\{l_j\}_{j=1}^{\infty}$, $\{m_j\}_{j=1}^{\infty}$ and where $\xi_j=\frac{p_{m_j}-p_{l_j}}{|p_{l_j}-p_{m_j}|}$. Applying the density of $\mathcal T \times \mathcal P$ in $(0,T)\times \p \Omega$, it follows that there exists a sub-sequence $\{j_\ell\}_{\ell=1}^{\infty}\subset \N$ such that 
$$\lim_{\ell\to\infty}(t_{k_{j_\ell}},p_{l_{j_\ell}})= \gamma(0)$$
and such that 
$$ \lim_{\ell\to\infty} \xi_{j_\ell}=\xi.$$
Thus, using continuity of $V_1-V_2$ together with \eqref{V_gamma}, it follows that
$$ \int_{\R} (V_1-V_2)(\gamma(s))\,ds=0$$
for all light rays $\gamma$ in $\mathcal D$. Finally, applying the injectivity of the light ray transform (see for example \cite[Theorem 2.1]{Ste}) we deduce that
$$ V_1 = V_2 \quad \text{on $\mathcal D$}.$$
This completes the proof of the theorem.
\end{proof}

\begin{proof}[Proof of Corollary~\ref{c1}] Let $f\in L^2(\R^{1+n})$ be the source term given by \eqref{universal_source} and let $V_j \in \mathcal C^4([0,T]\times\R^n)\cap \mathcal C([0,T];\mathcal C^4_0(\Omega))$, $j=1,2$. We consider also $u_j\in \mathcal C^1([0,T];L^2(\R^n)) \cap \mathcal C([0,T];H^1(\R^n))$ solving \eqref{pf} with $V=V_j$, $j=1,2$. Assuming that the condition
$$u_1(t,x)=u_2(t,x),\quad (t,x)\in(0,T)\times\mathcal O)$$
is fulfilled, we will prove that $V_1=V_2$ on $\mathcal D_{T_1}$. Consider $u=u_1-u_2$ and notice that $u$ satisfies
\begin{equation}\label{c1a}\left\{\begin{array}{ll}\partial_t^2u-\Delta_x u=0,\quad &\textrm{in}\ (0,T)\times(\tilde{\Omega}\setminus\overline{\Omega}),\\  u(0,\cdot)=\partial_tu(0,\cdot)=0,\quad &\textrm{in}\ \R^n,\\ u=0,\quad &\textrm{on}\ (0,T)\times\mathcal O.\end{array}\right.\end{equation}
Now let us consider $\tilde{u}$ defined on $[-T,T]\times \R^n$ by $\tilde{u}=u$  on $[0,T]\times \R^n$ and 
$$\tilde{u}(-t,x)=u(t,x),\quad (t,x)\in[0,T]\times \R^n.$$
Since $u(0,\cdot)=\partial_tu(0,\cdot)=0$, we deduce that $\tilde{u}\in H^1(\R^{1+n})$ and \eqref{c1a} implies
\begin{equation}\label{c1b}\left\{\begin{array}{ll}\partial_t^2\tilde{u}-\Delta_x \tilde{u}=0,\quad &\textrm{in}\ (-T,T)\times(\tilde{\Omega}\setminus\overline{\Omega}),\\   \tilde{u}=0,\quad &\textrm{on}\ (0,T)\times\mathcal O.\end{array}\right.\end{equation}
Applying the global Holmgren uniqueness theorem for hyperbolic equations (see e.g. \cite[Theorem 3.11]{KKL} or \cite[Theorem 2.2]{KMO} ), which is a consequence of the well known local unique continuation result of \cite[Theorem 1]{Ta}, we deduce that 
$$u(t,x)=\tilde{u}(t,x)=0,\quad t\in(0,T),\ x\in\{y\in\tilde{\Omega}\setminus\overline{\Omega}:\ \textrm{\emph{dist}}(y,\mathcal O)<T-t\}.$$
In particular, \eqref{c1a} implies that $u=0$ on $(0,T_1)\times (\tilde{\Omega}\setminus\overline{\Omega})$. Therefore, we have
$$u_1(t,x)=u_2(t,x),\quad (t,x)\in(0,T_1)\times(\tilde{\Omega}\setminus\overline{\Omega}))$$
and, repeating the arguments of Theorem \ref{t0} with $T$ replaced by $T_1$, we deduce that $V_1=V_2$ on $\mathcal D_{T_1}$. This completes the proof of the corollary.\end{proof}

------------------------------------------------------

\section*{Acknowledgments}
A.F acknowledges the support from the EPSRC grant EP/P01593X/1. Y.K. acknowledges the support from the Agence Nationale de la Recherche (project MultiOnde) grant ANR-17-CE40-0029.

\bibliographystyle{abbrv}
%\bibliography{main}

\end{document}